\documentclass[11pt,a4paper]{amsart}
\usepackage{lmodern}
\usepackage[T1]{fontenc}
\usepackage[utf8]{inputenc}
\usepackage{amsmath}
\usepackage{amsfonts}
\usepackage{amssymb}
\usepackage{amsmath}	
\usepackage{amsthm}
\usepackage{amssymb}
\usepackage{wasysym}
\usepackage{cite}	
\usepackage[shortlabels]{enumitem}	
\usepackage{xcolor}
\usepackage{tikz-cd}
\usepackage{float}
\usepackage{url}
\usepackage{hyperref}
\newtheorem{thm}{Theorem}[section]
\newtheorem{theorem}{Theorem}
\newtheorem{cor}[thm]{Corollary}
\newtheorem{lm}[thm]{Lemma}
\newtheorem{pr}[thm]{Proposition}
\theoremstyle{definition}
\newtheorem{df}[thm]{Definition}
\theoremstyle{remark}
\newtheorem{rem}[thm]{Remark}

\newcommand{\cftf}{$C(4)$--$T(4)$}
\newcommand{\cftfs}{$C(4)$--$T(4)$ }
\newcommand{\cs}{$C(6)$}
\newcommand{\css}{$C(6)$ }
\newcommand{\ctts}{$C(3)$--$T(6)$}
\newcommand{\cttss}{$C(3)$--$T(6)$ }
\newcommand{\aot}{\cs, \cftf, and \ctts}
\newcommand{\aots}{\cs, \cftf, and \cttss}

\usepackage[a4paper, left = 2.5cm, right = 2.5cm, top = 2.5cm, bottom = 2.5cm, headsep = 1.2cm]{geometry}
\begin{document}
	
	\title[The small cancellation flat torus theorem]{The small cancellation flat torus theorem}

	\author[Karol Duda]{Karol Duda$^{\dag}$}
	\address{Silesian University of Technology,
	Gliwice, Poland}
	\email{karol.duda@polsl.pl}

	\begin{abstract}
	We establish Flat Torus Theorem–type results for groups acting on small cancellation complexes satisfying \aots conditions.
For \cttss complexes the result closely parallels the CAT(0) setting.
For \css complexes we prove an analogous theorem using a refined notion of flat, exploiting the relationship between \css complexes and their duals.
In the \cftfs case we demonstrate that genuine flats do not necessarily exist, providing an explicit example of a \cftfs complex with an action of $\mathbb{Z}^2$ without invariant flat, and hence not admitting any CAT(0) metric invariant under automorpihsms. We introduce the notion of thickened-flats and prove a Flat Torus Theorem for thickened-flats by passing to quadric complexes via quadrization and invoking the Quadric Flat Torus Theorem of Hoda–Munro.
	\end{abstract}
	
	\maketitle

	\section{Introduction}
	The small cancellation theory is a classical powerful tool for constructing examples of groups with interesting features, as well as for exploring properties of well-known groups. It might be seen as a bridge between the combinatorial and the geometric group theories, and as one of the first appearances of nonpositive curvature techniques in studying groups. Constituting a classical part of mathematics, small cancellation techniques are still being developed, having numbers of variations, and have also been used in proving new important results nowadays. We direct the reader to the classical book of Lyndon-Schupp~\cite{ls} for basics on small cancellation.
	
	Among the many notions of small cancellation, the best known and the most deeply explored ones are the so-called ``combinatorial small cancellation conditions'': \aot .
	
	A recurring pattern in the study of groups acting on nonpositively curved spaces is a Flat Torus Theorem, which in the classical setting of Riemannian manifolds of nonpositive sectional curvature, says that any $\mathbb{Z}^n$ in the fundamental group stabilizes a \emph{flat}, i.e. an isometrically embedded copy of Euclidean $n$-space, in the universal cover.  This classical theorem has been generalized to CAT(0) spaces \cite{BH}. 
	
\begin{theorem}[Flat torus theorem]\label{prop:FPT}
Let $G$ be a free-abelian group of rank~$n$ acting metrically properly and semisimply on a CAT(0) space $\widetilde X$.
There exists a subspace $V\times F \subset \widetilde X$ with $F$ isometric to $\mathbb{E}^n$
such that  $G$ stabilizes $V\times F$ and acts as: \  $g(v,f)=(v,gf)$ for all $(v,f)\in V\times F$ and $g \in G$.
\end{theorem}

Further combinatorial analogues have been established for CAT(0) cube complexes \cite{wisewoodhouse}, systolic complexes \cite{Elsner_flattorus}, and other settings. Most recently Nima Hoda and Zachary Munro proved analogue for quadric complexes \cite{flatquadric}.

In this paper we prove versions of Flat Torus Theorem for \aot small cancellation complexes. There are substantial differences among these versions.

The \cttss case is the closest to the CAT(0) setting and allows us to obtain isometrically embedded copy of Euclidean $2$-space.
An action of a group $G$ on a metric space $X$ is \emph{metrically proper} if $\{g\in G:gB\cap B\neq \emptyset\}$ is finite for any metric ball.

\begin{theorem}[Flat Torus for \cttss complexes]\label{flattorus36}
Let $G$ be a non-cyclic free abelian group. Let $G$ act metrically properly and semisimply on a simply-connected \cttss small cancellation complex $X$. Then $G\cong \mathbb{Z}^2$ and there is a $G$-invariant flat in $X$.
\end{theorem}

The case of \css complexes is already more restricted.
While there exists an embedded \emph{flat plane} whose combinatorial structure is faithfully reflected in the \css small cancellation complex, the standard metric on a complex, where each edge has length~$1$ and each $n$–gon is a regular Euclidean polygon, is not necessarily sufficient for this embedding to be isometric.  Therefore it is not a flat in the same sense as in the \cttss case. The following theorem uses a slightly modified notion of a flat to account for this difference. 

Moreover, in the proof we use the systolic dual of the complex $X$ together with the Flat Torus Theorem for systolic complexes \cite[Theorem 6.1]{Elsner_flattorus}. That theorem assumes that the group acts by simplicial automorphisms on the systolic complex, so in our setting the group acts by automorphisms on the \css complex itself.

\begin{theorem}[Flat Torus for \css complexes]\label{flattorus6}
Let $G$ be a virtually non-cyclic free abelian group. Let $G$ act metrically properly on an simply-connected \css small cancellation complex $X$. Then $G\cong \mathbb{Z}^2$ and there is a $G$-invariant flat in $X$.
\end{theorem} 

The case of \cftfs complexes is significantly more complicated, since an isometrically embedded copy of Euclidean $2$-space need not exist. In Section 7 we present an example of such a \cftfs complex and observe that it cannot have CAT(0) metric invariant under automorphisms. Nevertheless, this example contains a certain quasi-flat we call \emph{thickened-flat}, and a version of the Flat Torus Theorem related to such thickened-flats still holds for \cftfs complexes.

\begin{theorem}[Flat Torus for \cftfs complexes]\label{flattorus44}
Let $G$ be a virtually non-cyclic free abelian group. Let $G$ act metrically properly on a simply-connected \cftfs small cancellation complex $X$. Then $G\cong \mathbb{Z}^2$ and there is a $G$-invariant thickened-flat in $X$.
\end{theorem}
   
\subsection*{Structure of the paper}
In Section \ref{sc} we give a brief introduction to small cancellation theory.
Section \ref{fl} focuses on flats in \css and \cttss complexes.
In Section \ref{cttss} we prove Theorem \ref{flattorus36}.
Section \ref{css} addresses the case of \css complexes. Subsection \ref{sys} provides a short introduction to systolic complexes, while Subsection \ref{flatcss} establishes a one-to-one correspondence between flats in \css complexes and flats in their systolic duals. We conclude this subsection by proving Theorem \ref{flattorus6}.

Section \ref{cftf} treats the case of \cftfs complexes. Subsection \ref{quad} gives a brief introduction to quadric complexes, while Subsection \ref{flatcftf} presents an example of a \cftfs complex without invariant flats and introduces the notion of thickened-flats. We conclude this subsection by proving Theorem \ref{flattorus44}.

\subsection*{Acknowledgements} The author would like to thank Martin Blufstein, Nima Hoda and Motiejus Valiunas whose comments led to many corrections and improvements.

\section{Basic definitions}
The purpose of this section is to give basic definitions and terminology regarding combinatorial $2$-complexes. We follow the notation of \cite{MW}.
For fundamental notions such as \textit{CW}-\textit{complexes}, \textit{nullhomotopy} and \textit{simple connectedness}, see Hatcher's textbook on algebraic topology \cite{AH}. Starting with section 3, we consider only $2$-dimensional CW-complexes and we will refer to them as ``$2$-complexes''. Throughout this section, we also mostly consider $2$-complexes, although some definitions are stated in full generality.

A map from $X$ to $Y$ is \textit{combinatorial} if it is a continuous map whose restriction to every open cell $e$ of $X$ is a homeomorphism from $e$ to an open cell of $Y$. A complex is \textit{combinatorial} if the attaching map of each of its $n$-cells is combinatorial for a suitable subdivision of the sphere $\mathbb{S}^{n-1}$. An \textit{immersion} is a combinatorial map that is locally injective.

Unless stated otherwise, all maps and complexes are combinatorial, and all attaching maps are immersions.

A \textit{polygon} is a $2$-disc with a cell structure that consists of $n$ $0$-cells, $n$ $1$-cells and a single $2$-cell. For any $2$-cell $C$ of a $2$-complex $X$ there exists a map $R\rightarrow X$ where $R$ is a polygon and the attaching map for $C$ factors as $\mathbb{S}^1\rightarrow \partial R \rightarrow X$. Because of that, by a \textit{cell}, we will mean a map $R\rightarrow X$ where $R$ is a polygon. 

A \textit{simplicial} complex is an $n$-dimensional complex with each $i$-cell being an $i$-simplex uniquely determined by $i+1$ distinct $0$-cells. In particular $1$-skeleton of simplicial complex is a simplicial graph. 

In part of this paper devoted to the \css complexes we will often work with \textit{triangle} complexes which are $2$-dimensional simplicial complexes. In the case of triangle complexes, instead of the usual terms, $0$-cell, $1$-cell and $2$-cell, we will use vertex, edge and triangle, respectively.

A \textit{path} in $X$ is a combinatorial map $P\rightarrow X$ where $P$ is either a subdivision of the interval or a single $0$-cell. For given paths $P_1\rightarrow X$ and $P_2\rightarrow X$ such that the terminal point of $P_1$ is equal to the initial point of $P_2$, their \textit{concatenation} is the natural path $P_1P_2\rightarrow X$. 
Similarly, a \textit{cycle} is a map $C\rightarrow X$ where $C$ is a subdivision of a circle $\mathbb{S}^1$. 
It is clear that a closed path, i.e. a path such that the initial point is equal to the terminal point, is a cycle.
We will often identify paths and cycles with their images in the complex $X$.

A \textit{disc diagram} is a contractible finite $2$-complex $D$ with a specified embedding into the plane. The \textit{area} of diagram $D$ is the number of its $2$-cells. If $D$ is a disc diagram, then the diagram $D$ \textit{in} $X$ is $D$ along with a combinatorial map from $D$ to $X$ denoted by $D\rightarrow X$. 

A \textit{boundary cycle} $\partial D$ of $D$ is the attaching map of the $2$-cell that contains the point $\infty$ when we regard $\mathbb{S}^2=\mathbb{R}^2\cup \{\infty\}$.

A \textit{graph} is a $1$-dimensional CW-complex with $0$-cells called \textit{vertices} and $1$-cells called \textit{edges}. Distances between vertices in graphs are always measured by the \textit{standard graph metric} which is defined for a pair of vertices $u$ and $v$ as the number of edges in the shortest path connecting $u$ and $v$. 

A \textit{link} of a $0$-cell $v$ of a $2$-complex $X$ (denoted $X_v$) is the graph whose vertices correspond to the ends of $1$-cells of $X$ incident to $v$, and an edge joins vertices corresponding to the ends of $1$-cells $e_1, e_2$ if and only if there is a $2$-cell $F$ such that $e_1$ and $e_2$ form a corner of $F$ at $v$.

\section{Small cancellation}\label{sc}

The following definition is crucial in the small cancellation theory.
\begin{df}
Let $X$ be a combinatorial $2$-complex. A non-trivial path $P\rightarrow X$ is a \textit{piece} of $X$ if there are $2$-cells $R_1$ and $R_2$ such that $P\rightarrow X$ factors as $P\rightarrow R_1\rightarrow X$ and as $P\rightarrow R_2\rightarrow X$ but there does not exist a homeomorphism $\partial R_1\rightarrow \partial R_2$ such that
there is a commutative diagram:

$$
  \begin{tikzcd}
    P \arrow{r} \arrow{d} & \partial R_2 \arrow{d} \\
    \partial R_1 \arrow{r} \arrow[swap]{ur} & X
  \end{tikzcd}
$$
\end{df}

Intuitively, a piece of $X$ is a path which is contained in boundaries of $2$-cells of $X$ in at
least two distinct ways.

\begin{df}
Let $X$ be a $2$-complex. A disc diagram $D\rightarrow X$ is \textit{reduced} if for every piece $P\rightarrow D$ the composition $P \rightarrow D\rightarrow X$ is a piece in the image of $D$ in $X$.  
\end{df}

The following lemma is known as the Lyndon-van Kampen lemma (see e.g. \cite[Lemma 2.17]{MW}).

\begin{lm}\label{lvk}
If $X$ is a $2$-complex and $P\rightarrow X$ is a nullhomotopic closed path, then there exists a reduced disc diagram $D\rightarrow X$ such that $P\rightarrow D$ is the boundary cycle of $D$, and $P\rightarrow X$ is the composition $P\rightarrow D\rightarrow X$.
\end{lm}

Let $\alpha$ be a nullhomotopic cycle in $X$. We say that disc diagram $D$ is a \textit{minimal area disc diagram for $\alpha$ in $X$} if $\alpha$ is a boundary cycle of $D$ and $D$ has the smallest area amongst all disc diagrams with boundary cycle $\alpha$. 

We will now define small cancellation conditions $C(p)$ and $T(q)$.

\begin{df}
Let $p\in\mathbb{N}$. We say that a $2$-complex $X$ satisfies the $C(p)$ \textit{small cancellation condition} provided that for each $2$-cell $R \rightarrow X$ its attaching map $\partial R\rightarrow X$ is not a concatenation of fewer than $p$ pieces in $X$.
\end{df}
In the following definition of $T(q)$ condition we will use the notion of \textit{valence} of a $0$-cell $v\in X$ in the complex $X$, i.e.~the number of ends of $1$-cells incident to it. We denote it by $\delta_{X}(v)$ and drop the subscript if it is clear from the context.
\begin{df}
Let $q\in\mathbb{N}$. We say that a $2$-complex $X$ satisfies the $T(q)$ \textit{small cancellation condition}  if there does not exist a reduced map $D\rightarrow X$ where $D$ is a disc diagram containing an internal $0$-cell $v$ such that $2<\delta(v)<q$.  
\end{df}

If a complex satisfies both $C(p)$ and $T(q)$ conditions, then we call it a $C(p)$--$T(q)$ complex. We will now state some important properties of small cancellation complexes.

We now list some well known properties of simply connected \cftfs and \css complexes (see e.g. \cite{hoda2019quadric}, \cite{OP18}).

\begin{lm}\label{lem:intersections}
Let  $X$ be a simply connected \cftfs or \css small cancellation complex.
Then the following hold:
  \begin{enumerate}
  \item  every $2$-cell $F$ of $X$ is embedded.
  \item the intersection of any pair of intersecting two $2$-cells $F_1,F_2$ of $X$ is connected.
  \item if $F_1,F_2,\ldots, F_n$ are pairwise intersecting $2$-cells of $X$. Then $F_1\cap F_2\cap\ldots\cap F_n$ is non-empty and connected.		
  \end{enumerate}
\end{lm}

The third property is known as Helly Property.
In the case of \cftfs complexes, the following, stronger version was obtained by Nima Hoda \cite[Proposition 3.8]{hoda2019quadric}.  

\begin{pr}[Strong Helly Property]\label{stronghellycf}
Let $F_1$, $F_2$ and $F_3$ be a pairwise intersecting $2$-cells of a simply connected \cftfs complex. Then the intersection of some pair of these $2$-cells is contained in the remaining $2$-cell, i.e.~for some permutation $\sigma$ of the indices

$$F_{\sigma(1)}\cap F_{\sigma(2)}\subset F_{\sigma(3)}.$$
\end{pr}

Note that \css complexes do not have Strong Helly property (think of regular tesselation of plane by hexagons).
On the other hand, if the intersection is a piece, a weaker version stil holds.

\begin{lm}\label{stronghellypiece}
Let $F_1$, $F_2$ and $F_3$ be a pairwise intersecting $2$-cells of a simply connected \css complex and let $F_1\cap F_2\cap F_3=p$ be a piece. Then the intersection of some pair of these $2$-cells is contained in the remaining $2$-cell, i.e.~for some permutation $\sigma$ of the indices

$$F_{\sigma(1)}\cap F_{\sigma(2)}\subset F_{\sigma(3)}.$$
\end{lm}

\begin{proof}
Assume that it is not true. 
Let $p_1,p_2$ be pieces adjacent to $p$ such that, $p_1$ is a piece between $F_1$ and $F_2$ that does not belong to $F_3$, and $p_2$ is a piece between $F_1$ and $F_3$ that does not belong to $F_2$.
Then $p\cup p_1\cup p_2$ is connected and clearly belongs to $F_1$. 
Since $p_1$ does not belong to $F_3$, $p_2$ does not belong to $F_2$ and all cells are embedded, then either $F_2\cap F_3=p\subset F_1$ or that intersection is disconnected, a contradiction.
\end{proof}

\begin{rem}
This lemma is not specific to the \css condition; it relies only on the facts that $2$-cells are embedded and that intersections of $2$-cells are connected. Both of those facts hold for simply-connected \css complexes by Lemma \ref{lem:intersections}. 
\end{rem}

	In the case of \cttss complexes, note the fact that all pieces of such complexes are short, which was first observed by Pride \cite{Pride} in the following Lemma.

\begin{lm}
If $X$ is a $T(q)$ complex for $q\geq 5$, then every piece in $X$ has length $1$.
\end{lm}

Using this fact it is possible to prove the following.

\begin{theorem}[Theorem 3 of \cite{cftf}
]\label{thm:tF}
	Let $X$ be a simply connected \cttss  small cancellation complex. Then there exists a metric on $X$ turning it into
	a CAT$(0)$ triangle complex $\mathfrak{X}$ such that every automorphism of $X$ induces an automorphism of $\mathfrak{X}$.
	\end{theorem}
	
Complex $\mathfrak{X}$ in the formulation is a triangle complex obtained from $X$ by taking a barycentric subdivision of each cell in $X$.

\section{Flats}\label{fl}

\begin{df}
Let $\mathbb{E}^2_{\triangle}$ denote the standard tiling of the Euclidean plane by equilateral triangles.
Let $X$ be a \cttss complex.
A flat in $X$ is an isometric embedding $\mathbb{E}^2_{\triangle}\rightarrow X$.
\end{df}

Although the same definition can be formulated using square or hexagonal tilings of the Euclidean plane, it is not suitable for $C(p)$--$T(q)$ small cancellation complexes when $q<5$.
 
Indeed, the conditions $C(p)$ (for any $p$) and $T(q)$ (for $q<5$) are preserved under subdivision of edges. 
Take a complex $X$ containing an isometrically embedded flat, and subdivide all edges belonging to a flat.
Such a subdivision does not change the fact that the flat is embedded, nor does it change the combinatorial structure reflected in the small cancellation complex.

However, the standard metric on a complex—where each edge has length~$1$ and each $n$–gon is a regular Euclidean polygon—makes it impossible to isometrically embed the entire $\mathbb{E}^2_{\hexagon}$ (resp. $\mathbb{E}^2_{\square}$) into $X$.
Moreover, even restricting attention to the $1$–skeleton (as is sometimes done, for example, in the systolic or quadric settings) is insufficient, since distances outside the flat may become shorter.

While one could choose a metric tailored to ensure the flat is isometrically embedded, this is both artificial and contrary to the natural viewpoint of small cancellation theory; moreover, we aim for a more general statement.

Instead, we use another natural notion of distance in small cancellation complexes, the \emph{gallery distance}. 

Given a pair of $2$-cells $F_1$ and $F_n$ in $X$ we say that $\mathfrak{G} :=\{F_1,\ldots,F_n\}$ is a \emph{gallery} between $F_1$ and $F_n$ in $X$ if $F_1\ldots F_n$ are distinct $2$-cells, each satisfying $F_i \cap F_{i+1} \neq \emptyset$ for $1 \le i \le n-1$, and if there is no smaller subfamily containing both $F_1$ and $F_n$ with this property.
 
The \emph{length} of a gallery $\mathfrak{G}$ is the number of $2$–cells it contains, denoted $l(\mathfrak{G})$.
The \emph{gallery distance} between $F_1$ and $F_2$ is the minimum of $l(\mathfrak{G})$ over all such galleries.

Following Wise \cite{Wise2004CubulatingSC} we define flat plane as a \css complex.

\begin{df}
A \emph{flat plane} $\mathfrak{F}_{\hexagon}$ is a \css complex that is homeomorphic to the Euclidean plane and has property that each closed $2$-cell intersects exactly six distinct neighboring closed $2$-cells. For a flat plane $\mathfrak{F}_{\hexagon}$ there is an associated flat plane $\mathfrak{F}'_{\hexagon}$ obtained by removing $0$-cells of valence $2$. The resulting complex $\mathfrak{F}'_{\hexagon}$ is a regular hexagonal tiling of the Euclidean plane. 
\end{df}

In the proof of the following proposition we will use numberings of the $2$-cells of a complex. The \emph{numbering} of the set $X_2$ of all $2$-cells of a complex $X$ is a bijection $\phi:\mathbb{N}\rightarrow X_2$.

\begin{pr}\label{pr:flat_planecs}
Let $X$ be a simply connectd \css complex without free $1$-cells (i.e.  there is no $1$-cell not belonging to any $2$-cell). 
If for every $2$-cell $F\in X$ the following hold: 
\begin{enumerate}
\item $F$ has non-empty intersection with exactly six other $2$-cells $F_1,\ldots F_6$ (ordered clockwise) and $F\cap (F_1\cup\ldots\cup F_6)=\partial F$.
\item $F\cap F_i\cap F_j$ is either empty or a single $0$-cell.
\item $F\cap F_i\cap F_j\cap F_k$ is empty for all triples $(i,j,k)$.
\end{enumerate}
then $X$ is a flat plane.
\end{pr}

\begin{proof}
Note that any flat plane $\mathfrak{F}_{\hexagon}$ clearly satisifes Conditions (1)-(3).

The second part of Condition (1), together with \css condition implies that $F\cap F_i$ is a piece for every~$i$. 
Since there are no free $1$-cells, the second part of Condition (1) further implies that every $1$-cell of $X$ is internal.
Moreover, every $0$-cell of $X$ is internal as well. Indeed, since $X$ is simply connected and has no free $1$-cells, each $0$-cell belongs to some $2$-cell $F$. If a $0$-cell $v$ were non-internal, Condition~(1) would force $F\cap F_i = {v}$, a contradiction.

By Condition (3) every point of $\partial F$ belongs to at most two of $F_1,\ldots F_6$.
Combined with Condition~(2) and the fact that all $1$-cells and $0$-cells are internal, this implies that any point of $\partial F$ lying in three $2$–cells must be a $0$-cell of degree~3.
In particular, the link of such a $0$-cell is an embedded $3$–cycle.

It follows that every $1$-cell and every degree $2$ $0$-cell lies in exactly one intersection of the form $F\cap F_i$.

We now consider a numbering of the complex $X$. Let $C_0$ be a $2$-cell of $X$, and set $\phi(0)=C_0$. Define 
$$A_0=\{C \in X_2\setminus C_0\ : C\cap C_0 \neq \emptyset\}.$$ 
Choose $C_1 \in A_0$, set $\phi(1)=C_1$, and define 
$$A_{0,1}=\{C \in A_0\setminus C_1 : C\cap C_0\cap C_1 \neq \emptyset\}.$$ 
Condition (1) implies that $|A_0|=6$. Moreover, by Condition (2), each triple intersection of $2$-cells is a single $0$-cell. 
Together with the fact that each intersection of pair of cells is a piece, this implies that $|A_{0,1}|=2$. 
Denote the cells in $A_{0,1}$ by $C_2,C_3$, and set $\phi(2)=C_2$, $\phi(3)=C_3$.

Inductively, assume that $\phi$ has been defined on a subset of cardinality $k+1$, so $\phi(0)=C_0, \ldots, \phi(k)=C_k$.
For any $0\leq i\leq k$ define 
$$A_i=\{C \in X_2\setminus C_i\ : C\cap C_i \neq \emptyset\}$$
and for any $0\leq i <j\leq k$ define
 $$A_{i,j}=\{C \in A_i\setminus C_j : C\cap C_i\cap C_j \neq \emptyset\}.$$
For all $i,j$, we have $|A_i|=6$ and $|A_{i,j}|=2$.
Let $i$ be the smallest index such that $A_i\cap \{C_0,\ldots, C_k\}\neq A_i$.
Choose the smallest $j>i$ such that $A_{i,j}\cap \{C_0,\ldots, C_k\}\neq A_{i,j}$. Observe that $|A_{i,j}\cap \{C_0,\ldots, C_k\}|\geq 1$ for For $k>1$. Indeed, for each $k>1$ the cell $C_k$ is chosen from some $A_{i,j}$, where $i<j\leq k$, thus $C_j\in A_{i,k}$.
Let $C_{k+1}$ be the remaining cell from $A_{i,j}$ and set $\phi(k+1)=C_{k+1}$. 
This shows that numbering $\phi$ is uniquely determined by the choice of the first three cells.

Construct a numbering $\psi$ of the $2$-cells of $\mathfrak{F}_{\hexagon}'$ in the same way. Note that for any $I\subset\mathbb{N}$, $\bigcap\limits _{i\in I} \psi(i)\neq\emptyset$ if and only if $\bigcap\limits _{i\in I}\phi(i)\neq\emptyset$.

For any $2$-cell of $X$ there is a natural homeomorphism between it and a $2$-cell with the same number in $\mathfrak{F}_{\hexagon}'$.
Consider the bijection between $X$ and $\mathfrak{F}_{\hexagon}'$ obtained by taking these homeomorphisms so they agree on intersections of $2$-cells.
This map is clearly continuous, therefore $X$ is a flat planes.
\end{proof}

We can now define flats in \css complexes.

\begin{df}
Flat plane $\mathfrak{F}_{\hexagon}$ is a \emph{flat} in \css complex $X$ if $\mathfrak{F}_{\hexagon}\subset X$ and for any pair $(F_1,F_2)\in \mathfrak{F}_{\hexagon}$ where $F_1,F_2$ are $2$-cells in $\mathfrak{F}_{\hexagon}$ the gallery distance in $\mathfrak{F}_{\hexagon}$ equal the gallery distance between the same cells in $X$.
\end{df}

\section{Proof of Theorem \ref{flattorus36}}\label{cttss}

\begin{proof}
By Theorem \ref{thm:tF} the barycentric subdivision $\mathfrak{X}$ of $X$ can be induced with a metric $\mathfrak{d}$ that makes it CAT(0). This metric is obtained by taking each triangle to be Euclidean triangle with angle $\frac{\pi}{2}$ adjacent to a center of an $1$-cell of $X$, angle $\frac{\pi}{3}$ adjacent to a center of a $2$-cell of $X$ and angle $\frac{\pi}{6}$ adjacent to a $0$-cell from $X$.

By the Flat Torus theorem for CAT(0) spaces (Theorem \ref{prop:FPT}), if a free-abelian group of rank $n\geq 2$ acts metrically properly and semisimply on a CAT(0) space then there exists a subspace $V\times F$ with $F$ isometric to $\mathbb{E}^n$.
Since dimension of $\mathfrak{X}$ is $2$, then $n=2$ and $V\times F\cong F\cong \mathbb{E}^2$.

It remains to show that $F$ is also a flat in the original complex $X$.
Because $F$ is a flat in $\mathfrak{X}$, every vertex of $F$ in $\mathfrak{X}$ has link length exactly $2\pi$.

Consider first vertices arising from the centers of $1$–cells of $X$.
Such a vertex can belong to exactly two triangles of $\mathfrak{X}$, hence to exactly two $2$–cells of $X$, and those two $2$–cells must intersect along a piece.

If the vertex is a $0$–cell of $X$, then for its link to have length $2\pi$ it must lie in exactly twelve triangles of $\mathfrak{X}$, and therefore in exactly six $2$–cells of $X$ contained in $F$.

Vertices corresponding to centers of $2$–cells of $X$ can have link length $2\pi$ only if they lie in exactly six triangles of $\mathfrak{X}$—that is, only if the original $2$–cell is itself a triangle.

Moreover, the metric $\mathfrak{d}$ restricted to this flat $F$ agrees with the standard Euclidean metric. 
Thus $F$ corresponds to the standard tiling of the Euclidean plane by equilateral triangles.
\end{proof}

\section{Case of \css complexes}\label{css}
\subsection{Systolic complexes and nerve complex}\label{sys}

In this section we define systolic complexes and state several results concerning them and their relation to \css complexes.
The study of groups acting on systolic complexes were introduced by T. Januszkiewicz and J. Świątkowski \cite{JŚ06} and independently by Haglund \cite{syshag}.
Systolic complexes might be seen as a simplicial version of the cubical combinatorial nonpositive curvature theory of $\mathrm{CAT(0)}$ cube complexes. We folllow here the notation of \cite{JŚ06}.

\begin{df}
Let $X$ be a simplicial complex. 
$X$ is $k$-\textit{large} if $X$ is flag and every cycle in $X$ of length less than $k$ has a diagonal, i.e. in $X$ there is an edge connecting nonconsecutive vertices of the cycle.
\end{df}

\begin{lm}\cite[Lemma 1.3]{JŚ06}\label{13tjjs}
Suppose that $X$ is $k$-large and $\mathbb{S}^1_m$ denotes the triangulation of $\mathbb{S}^1$ with $m$ $1$-cells. If $m<k$ then any simplicial map $f:\mathbb{S}^1_m\rightarrow X$ extends to a simplicial map from disc $\mathbb{D}^2$, triangulated so that the triangulation on the boundary is $\mathbb{S}^1_m$ and so that there are no interior vertices in $\mathbb{D}^2$.
\end{lm}

\begin{df}
A simplicial complex $X$ is \textit{systolic} if $X$ is simply connected, and links of all vertices of $X$ are $6$-large.
\end{df}

Let $X$ be a $2$-complex with embedded $2$-cells, such that every $1$-cell is contained in the boundary of some $2$-cell.
Let $\mathcal{U}$ be a cover of $X$.
The \textit{nerve} of the cover $\mathcal{U}$ is a triangle complex whose vertex set is $\mathcal{U}$, and vertices $U_0,\ldots, U_n$ span an $n$-simplex if and only if $\bigcap\limits_{0\leq i\leq n}U_i\neq \emptyset$.

If $X$ is a simply connected \css complex, then the nerve of the cover of $X$ by closed $2$-cells is systolic \cite[Theorem 10.6]{wise}; see also \cite{OP18}.

\begin{df}
A flat in a systolic complex $X$ is an isometric embedding of a $1$-skeleton of $\mathbb{E}^2_{\triangle}$ into $X$.
\end{df}

Elsner \cite{Elsner_flattorus} proved that uniformly locally finite systolic complexes satisfy an analogue of the Flat Torus Theorem.
While proving analogous result for quadric complexes, Hoda and Munro managed to omit this restriction to the case of uniformly locally finite complexes by passing to metrically proper actions. Their argument also holds in the systolic case hence we obtain the following result.

\begin{theorem}[Systolic Flat Torus Theorem]\label{thm:els}
Let $G$ be a non-cyclic free abelian group. Let $G$ act metrically properly by simplicial automorphisms on a systolic complex $X$. Then $G\cong \mathbb{Z}^2$ and there is a $G$-invariant flat in $X$.
\end{theorem}

\begin{proof}
Original proof of Elsner needs systolic complex $X$ to be uniformly locally finite in its first step, to prove that there exists a finite-index subgroup $H\mathbb{Z}^{2}$ and (possibly not systolic) triangulation of an $H$-invariant plane $\widetilde{T}$, such that every non-trivial element of $H$ has translation length at least $3$.

Suppose $\mathbb{Z}^2$ acts metrically properly (hence freely) on a systolic complex $X$ with quotient $\phi: X\rightarrow X/\mathbb{Z}^2$.
By \cite[Theorem C of Appendix]{flatquadric}, there exists a map from a torus $f: T\to X/\mathbb{Z}^2$ so that $f_{\star}:\pi_1 T \to \pi_1 X/\mathbb{Z}^2$ is an isomorphism. Let $d$ be the pseudometric on $\widetilde{T}$ induced by pulling back the metric on $X$ via the $\mathbb{Z}^2$-equivariant map $\tilde{f}:\widetilde{T}\to X$. The action of $\mathbb{Z}^2$ on $(\widetilde{T},d)$ is metrically proper and is cofinite on $\widetilde{T}^{(0)}$. Thus, letting $x_1,\ldots x_k\in \widetilde{T}^{(0)}$ be a set of orbit representatives, there are finitely many elements $g\in \mathbb{Z}^2$  that satisfy $d(x_i,gx_i)\leq 3$ for some $i\in \{1,\ldots, k\}$. A finite-index subgroup $H<\mathbb{Z}^2$ that omits all such $g$'s, has translation length of at least $3$ for any $h\in H\setminus\{\mathrm{id}\}$.

The rest of the proof follows the proof of Elsner.
\end{proof}
 
Observe that a $6$-cycle in a systolic complex can be filled in multiple ways by six triangles sharing common internal vertex (see Figure \ref{fig:trianglegrid}).

\begin{figure}[H]
\begin{center}
\includegraphics[scale=1.2]{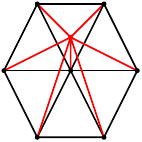}
\end{center}
\caption{$6$-cycle in a systolic complex filled by two discs consisting of six triangles.}\label{fig:trianglegrid}
\end{figure}

As a consequence, flats in systolic complex may differ by a single vertex. Therefore for any $G$-invariant flat in systolic complex, there may exist another $G$-invariant flat obtained by small deformations applied equivariantly under the $G$-action.

\begin{df}\label{thickflatsys}
Two flats $F$, $F'$ are \emph{equivalent} if they are at finite Hausdorff distance.
Let $F \colon \mathbb{E}^2_{\triangle} \to X$ be a flat in a quadric complex. The \emph{thick flat} $\mathrm{Thick}(F)$, is the full subcomplex spanned on all flats at finite Hausdorff distance from $F$.
\end{df}

As a corrolary of his results, Elsner further obtained the following version of the Flat Torus Theorem applicable to virtually free abelian groups and thick flats.

\begin{thm}[Corrolary 6.2. of\cite{Elsner_flattorus}]\label{virtflatsys}
    Suppose a virtually $\mathbb{Z}^2$ group $G$ acts metrically properly on a systolic complex $X$. Then there is a $\mathbb{Z}^2$-invariant flat in $F$ in $X$, such that $\mathrm{Thick}(F)$ is $G$-invariant.
\end{thm}

For the main result in \css case, we require a final characterization of equivalent flats, originally observed by Elsner. Note that we adopt a modified version of the notation introduced by  Hoda-Munro.
We say that two distinct flats $F, F'$ in a systolic complex $X$ \emph{differ by vertex move} if $F'$ can be obtained from $F$ by replacing a single vertex $v\in F$ by $w\in F'$.

\begin{lm}\label{equivalentFlatssys}
    Suppose that $F$ and $F'$ are equivalent. Then there exists  a sequence of flats $F = F_0, F_1, F_2, \ldots$ such that $F_{i+1}$ differs from $F_i$ by a vertex move, for each $i$, and $F_i \to F'$ pointwise as $i\to \infty$.
\end{lm}

\if
Unfortunately such nerve complex provides less information about vertices of the original \css complex than quadrization about vertices of the \cftfs complex, namely it does not provide information about the gallery distance between vertices. Therefore we define analogue of quadrization, called systolization.

Let $X$ be a $2$-complex with a boundary of each $2$-cell being an embedded cycle. Let $X_0, X_2$ be the sets of $0$-cells and $2$-cells of $X$. Let $\Gamma_X$ be a graph on the vertex set $X_0\cup X_2$ where an edge joins $v\in X_0$ with $F\in X_2$ whenever $v\in \partial F$, and $v_1,v_2\in X_0$ whenever $(v_1,v_2)\in \partial F_1\cap \partial F_2$. 

\begin{df}
The \textit{systolization} $Y$ of a complex $X$ is the $3$-flag completion $Y=\overline{\Gamma_X}$ i.e.~a complex obtained from $\Gamma_X$ by spanning a $2$-cell on each of its nontrivial $3$-cycles. 
\end{df}

If we additionally assume that every $1$-cell of $X$ is contained in the boundary of a $2$-cell then simply connectedness of $X$ implies simply connectedness of its systolization. It is clear as quadrization of such complex is simply connected, and systolization can be obtained from the quadrization by subdividing each square into two triangles, which clearly does not break simply connectedness, and then spanning triangle on each $3$-cycle that was constructed in the process.

\begin{lm}
Let $X$ be a simply connected $2$-complex. If $X$ is \css then its systolization $Y$ is systolic.
\end{lm}

\begin{proof}
Since $X$ is simply connected $Y$ is also simply connected. By the construction of systolization, link of each vertex is flag, and it suffices to show that every cycle of length less than six has a diagonal.

Observe that whenever in a cycle in the link there is a vertex corresponding to some $2$-cell $F$, then two of edges in a cycle join $F$ with vertices at its boundary.
Moreover any cycle of length less than six joining only vertices of $X$ corresponds to a tree in $X^1$. Otherwise there would be a $2$-cell bounded by a cycle containing of less than six pieces.  

An embedded cycle of length $4$ corresponds either to
\begin{itemize}
\item pair of $2$-cells intersecting along a piece;
\item single $2$-cell $F$ and three vertices, one of them possibly outside $\partial F$;
\item four vertices.
\end{itemize}

There is a diagonal in each case.

If a pair of $2$-cells intersects along a piece, then two vertices coming from the end ofthe piece are joined by an edge.

A single $2$-cell $F$ is joined by an edge to all vertices in its boundary. If one of vertices, say $x_3$ does not belong to $F$ then $(x_1,x_2)$ and $(x_2,x_3)$ are pieces of $X$. If $x_1$ and $x_3$ are not connected, then there is no piece containing them both.
There are two nullhomotopic cycles $(x_1,x_2,x_3)$ in $X$, going through two sides of boundary of $F$ between $x_1$ and $x_3$. 
Let $D$ be a minimal area disc diagram in $X$ bounded by any of $(x_1,x_2,x_3)$. It consists of a single cell or at least two shells. A priori $F$ might be a cell depending on the choice of cycle, but it cannot be a single cell as then it contains $x_2$ and if there is another cell then it is possible to show that it is bounded by three cells at most.

Vertices $(x_1,x_2,x_3,x_4)$ form a cycle if there are pieces $(x_1,x_2),(x_2,x_3),(x_3,x_4),(x_4,x_1)$ in $X$. There are no diagonal only if these pieces form a non nullhomotopic cycle in $X$. Let $D$ be a minimal area disc diagram in bounded by $x_1,x_2,x_3,x_4$. It consists of a single cell or at least two shells. But it is impossible as it is bounded by four pieces at most.

An embedded cycle of length $5$ corresponds either to 
\begin{itemize}
\item pair of $2$-cells intersecting along a piece and a single vertex that belongs to at least one of them;
\item single $2$-cell $F$ and four vertices, two of them possibly outside $\partial F$;
\item five vertices.
\end{itemize}

There is a diagonal in each case.
Again if a pair of $2$-cells intersects along a piece, then two vertices coming from the end ofthe piece are joined by an edge.

A single $2$-cell $F$ is joined by an edge to all vertices in its boundary. If one of vertices, say $x_3$ does not belong to $F$ then obvious diagonal is $(F,x_2)$. Thus two vertices does not belong to $F$, say $x_2$ and $x_3$. Then $(x_1,x_2)$, $(x_2,x_3$ and $(x_3,x_4)$ are pieces of $X$.

There are two nullhomotopic cycles $(x_1,x_2,x_3,x_4)$ in $X$, going through two sides of boundary of $F$ between $x_1$ and $x_4$. 
Let $D$ be a minimal area disc diagram in $X$ bounded by any of $(x_1,x_2,x_3,x_4)$. It consists of a single cell or at least two shells. A priori $F$ might be a cell depending on the choice of cycle, but it cannot be a single cell as then it contains $x_3$ and $x_4$ and if there is another cell then it is possible to show that it is bounded by four cells at most.

If vertices $(x_1,x_2,x_3,x_4,x_5)$ form a cycle then there are pieces $(x_1,x_2),(x_2,x_3),(x_3,x_4),(x_4,x_5),(x_5,x_1)$ in $X$. There are no diagonal only if these pieces form a non nullhomotopic cycle in $X$. Let $D$ be a minimal area disc diagram in bounded by $(x_1,x_2,x_3,x_4,x_5)$. It consists of a single cell or at least two shells. But it is impossible as it is bounded by five cells at most.
\end{proof}

\fi

\subsection{Flat torus for \css complexes}\label{flatcss}

We begin by observation that flats in systolic duals to \css small cancellation complexes are unique. 

\begin{lm}\label{lm:equivflatdiffsys}
Let $X$ be a simply-connected \css complex and $Y$ its nerve complex.
Let $F, F'$ be flats in $Y$. If $F$ and $F'$ are equivalent then $F=F'$.
Consequently $\mathrm{Thick}(F)=F$.
\end{lm}

\begin{proof}
Assume that $F'$ is equivalent to $F$. By Lemma \ref{equivalentFlatssys} there exists a sequence of flats $F_1,F_2 \ldots$ such that each $F_{i+1}$ differs from $F_i$ by a vertex move. Hence without loss of generality we may assume that $F'$ differs from $F$ by a single vertex move replacing vertex $v$ with vertex $w$, both adjacent to the $6$-cycle $\{v_1,v_2,v_3,v_4,v_5,v_6\}$.
Vertices $v,w,v_1,\ldots, v_6 $ correspond to $2$-cells $F_v, F_w, F_1,\ldots, F_6$.

By Lemma \ref{lem:intersections}, the intersection of two $2$-cells is connected, thus intersection of $F_v$ and $F_w$ is either a piece or a single vertex. Moreover, by the same lemma, any intersection $F_v\cap F_w\cap F_i\cap F_{i+1}$ is also connected and at least a vertex.

We claim that $F_v\cap F_w$ is a piece. Indeed, if it were a single vertex ,then each intersection $F_v\cap F_w\cap F_i$ would coincide with that same vertex for each $i$, implying that $F_i\cap F_j\neq\emptyset$ for all $1\leq i,j\leq 6$. 
Consequently $F$ would fail to be a flat, as all vertices in a cycle $\{v_1,v_2,v_3,v_4,v_5,v_6\}$ are adjacent to each other, contradicting the fact that $F$ is isometrically embedded.

Since each intersection $F_v\cap F_w\cap F_i$ and $F_v\cap F_w\cap F_i\cap F_{i+1}$ is non-empty and connected, $F_w$ covers at least four consecutive pieces $F_i\cap F_v$.
 
Because $X$ is \css complex, the boundary $\partial F_v$ cannot be entirely covered by pieces $F_v\cap (F_1\cup \ldots  \cup F_6)$. These pieces therefore form a closed path in $\partial F_v$, traversed in both directions by those six pieces.

It follows that some $F_i$ intersects some $F_j$ with $j\neq i\pm 1 \pmod 6$, implying that $v_i$ and $v_j$ are adjacent in $Y$.
But distance between $F_i$,$F_j$ in $F$, for $j\neq i\pm 1 \pmod 6$ is at least $2$, contradicting the facy that $F$ is isometrically embedded.
\end{proof}

\begin{lm}\label{flatequiv}
Let $X$ be a simply-connected \css small cancellation complex and $Y$ be its nerve complex. Let $E$ be a subcomplex of $X$ and $E_Y$ be its nerve complex.
Then $E$ is a flat in $X$ if and only if $E_Y$ is a flat in $Y$.  
\end{lm}

\begin{proof}
First observe that every gallery in $X$ corresponds to a path of the same length in $Y$. 
In particular, the gallery distance between two $2$-cells in $X$ equals the distance between vertices corresponding to these $2$-cells in the nerve of $X$.
Consequently, it is impossible for only one of $E$ or $E_Y$ to be isometrically embedded, and it remains to verify that $E$ is a flat plane if and only if $E_Y$ is a flat.

\medskip
\noindent($\Leftarrow$)
If $E$ is a flat plane, then its nerve $E_Y$ is clearly a copy of $\mathbb{E}^2_{\triangle}$ embedded in $Y$.

\medskip
\noindent($\Rightarrow$)
Assume now that $E_Y$ is a flat in $Y$. Let $E$ be a subcomplex of $X$ consisting of all $2$-cells of $X$ corresponding to vertices of $E_Y$.
Denote by $F_i$ the $2$–cell of $X$ corresponding to a vertex $x_i\in Y$ (possibly without index). It is clear that $E$ does not contain free $1$-cells. We will show that $E$ satisfies Conditions (1)-(3) of Proposition \ref{pr:flat_planecs}.

\medskip
\noindent\emph{Condition (1).} and the fact that each edge belongs to a piece. 
Observe that boundary of any $2$-cell of $E$ is fully covered by six other cells $F_1,\ldots F_6$. Take a vertex $x$ in $E_Y$. It is adjacent to exactly $6$ other vertices $x_1,\ldots x_6$ forming a cycle $(x_1,\ldots x_6)$.
No additional $2$–cell $F_7$ may intersect $F$: otherwise $x_7$ would be adjacent to $x$ in $Y$ while their distance in $E_Y$ is at least $2$, contradicting again that $E_Y$ is isometrically embedded in $Y$.
Suppose that $\partial F$ is not fully covered in $E$ by pieces $F\cap (F_1\cup \ldots  \cup F_6)$. Then these pieces form a closed path in $\partial F$, which is traversed by these six pieces in both directions.
Hence some $F_i$ intersects some $F_j$, $j\neq i\pm 1 \pmod 6$, implying that $x_i$ and $x_j$ are adjacent in $Y$.
But distance between $F_i$,$F_j$, for $j\neq i\pm 1 \pmod 6$ is at least $2$.
But then the distance between $F_i$ and $F_j$ in $E_Y$ is $2$, whereas in $Y$ it is $1$, contradicting the assumption that $E_Y$ is isometrically embedded.
Thus Condition (1) of Proposition \ref{pr:flat_planecs} is satisfied.

\medskip
\noindent\emph{Condition (2).}
Let $F_1,F_2,F_3$ be pairwise intersecting $2$–cells in $E$.
By Lemma \ref{lem:intersections}, their intersection is nonempty, and in $E$ they share either a common vertex or a piece.
By Lemma \ref{stronghellypiece}, we may assume $F_1\cap F_2\subseteq F_3$.
Since $X$ satisfies the \css condition, the boundary of $F_1$ is covered by six pieces.
But $F_1\cap F_2\subseteq F_1\cap F_3$ contradicts the fact that $F_1$ intersects exactly six $2$–cells.

\medskip
\noindent\emph{Condition (3).} Any nonempty intersection of $n$ $2$–cells in $X$ corresponds to an $n$–simplex in $Y$.
Thus if four $2$–cells from $E$ have nonempty intersection, they span a $4$–simplex in $Y$.
But then some pair of vertices of this simplex are adjacent in $Y$ but not in $E_Y$, contradicting the fact that $E_Y$ is isometrically embedded.

Consequently all conditions of Proposition \ref{pr:flat_planecs} hold, and $E$ is a flat plane. 
\end{proof}

\begin{proof}[Proof of Theorem \ref{flattorus6}]
Let $X$ be a \css complex, and $Y$ be a nerve complex of $X$.
Note that automorphism of a \css complex induces a symplicial automorphism of $Y$. 
Let $G$ be a virtually non-cyclic free abelian group and $H$ be its non-cyclic free abelian subgroup of finite index.
By Theorem \ref{thm:els}, it is clear that $H\cong \mathbb{Z}^2$ and there is a $H$-invariant flat $E_Y$ in $Y$.
By Lemma \ref{virtflatsys} $\mathrm{Thick}(E_Y)$ is $G$-invariant, and as a consequence of Lemma \ref{lm:equivflatdiffsys} $\mathrm{Thick}(E_Y)=E_Y$.
By Lemma \ref{flatequiv} it corresponds to a flat $E$ in $X$. It is clear that $E$ is also $G$-invariant.
\end{proof}

\section{The case of \cftfs complexes}\label{cftf}
\subsection{Quadric complexes and quadrization}\label{quad}
In this section we define quadric complexes and the quadrization of a complex. Quadric complexes were introduced by Nima Hoda in \cite{hoda2019quadric} as square analogues of systolic complexes. Much like the relationship between \css complexes and systolic complexes, quadric complexes arise naturally as duals, via the \emph{quadrization}, of \cftfs complexes.

\begin{df}\cite[Definition 1.1.]{hoda2019quadric}
A \textit{locally quadric complex} is a square complex $Y$ satisfying the following conditions.
\begin{enumerate}[(A)]
\item The attaching map of every square is an immersion.
\item If there are two squares $F_1,F_2\in Y$ that share at least three edges, then $\partial F_1=\partial F_2$.
\item If there are two squares $F_1,F_2\in Y$ such that $\partial(F_1\cup F_2)$ is a cycle of length $4$, then there is $F\in Y$ such that $\partial F=\partial(F_1\cup F_2)$.
\item If there are three squares $F_1,F_2,F_3\in Y$ such that $\partial(F_1\cup F_2\cup F_3)$ is a cycle of length $6$, then there exist $F,F'\in Y$ such that $\partial(F_1\cup F_2\cup F_3)=\partial(F\cup F')$, i.e.~this cycle has a diagonal that divides it into two $4$-cycles.
\end{enumerate}
A \emph{quadric complex} is a simply connected locally quadric complex.
\end{df}

\begin{figure}[h]
\begin{center}
\includegraphics[scale=2.2]{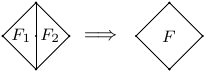}
\end{center}
\begin{center}
(C)
\end{center}
\begin{center}
\includegraphics[scale=2.2]{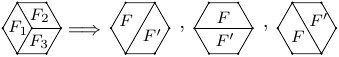}
\end{center}
\begin{center}
(D)
\end{center}
\caption{Condition (C) and (D) of Definition 7.1}
\end{figure}

For the remainder of this section, we will assume that every $1$-cell of a $2$-complex $X$ is contained in the boundary of at least one $2$-cell.
This is not a serious restriction, since each $1$-cell of $X$ not appearing on the boundary of a $2$-cell can be subdivided (if it is not embedded) and then thickened to a $2$-cell to obtain a new $2$-complex $X'$ which deformation retracts to $X$.  
The original complex $X$ embeds $\mathrm{Aut}(X)$-equivariantly into $X'$ via a continuous map $X \to X'$, which thus faithfully preserves group actions. Furthermore, if $X$ is \cftfs then so is $X'$.

Let $X$ be a $2$-complex with embedded cells and $X_0$ and $X_2$ be sets of its $0$- and $2$-cells.
The \emph{quadrization} of $X$ is the bipartite graph whose vertex set is $X_0\cup X_2$, with an edge joining $v\in X_0$ to $F\in X_2$ whenever $v\in \partial F$.

In \cite[Lemma 3.9 ]{hoda2019quadric} Hoda proved that the quadrization of simply-connected \cftfs complex is a quadric complex. 

As in the systolic case, a \emph{flat} in a quadric complex $X$ is an isometric embedding of the $1$-skeleton of $\mathbb{E}^2_{\square}$ into $X$, where $\mathbb{E}^2_{\square}$ denotes the standard tiling of the Euclidean plane by squares. Recently Hoda and Munro proved the following:

\begin{theorem}[Quadric Flat Torus Theorem]\cite[Theorem I]{flatquadric}\label{quadflat}
    Let $G$ be a non-cyclic free abelian group.  Let $G$ act metrically properly on a quadric complex $X$ (e.g. let $G$ act freely on a locally finite quadric complex $X$).  Then $G \cong \mathbb{Z}^2$ and there is a $G$-invariant flat in $X$.
\end{theorem} 

Thus it is natural to expect an analogous result for \cftfs complexes.
Before moving to this topic, let us state few additional results of Hoda and Munro, that will be relevant for this paper.
First, observe that a boundary of a $2\times 2$ square grid in a quadric complex can be filled in multiple ways by $2\times 2$ square grids (see Figure \ref{fig:squaregrid}).

\begin{figure}[H]
\begin{center}
\includegraphics[scale=1]{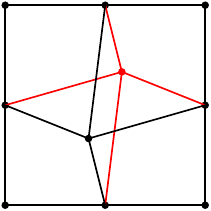}
\end{center}
\caption{Two $2\times 2$ square grids filling the same $8$-cycle in quadric complex.}\label{fig:squaregrid}
\end{figure}

As a consequence flats in quadric complexes may differ by change of one square $2\times 2$ grid to another.

\begin{df}\label{df:equivflats}
    Let $F:\mathbb{E}^2_{\square}\to X$ be a flat in $X$, and let $v$ be a vertex of $\mathbb{E}^2_{\square}$. Suppose a vertex $w$ of $X$ is adjacent to the images of the neighbors of $v$. Define $F':\mathbb{E}^2_{\square}\to X$ by $F'(v)=w$ and $F'(x)=F(x)$ for all $x\neq v$. Then $F'$ is a flat, and we say $F'$ differs from $F$ by a \emph{vertex move} on $v$. Two flats $F$, $F'$ are \emph{equivalent} if there exists a sequence of flats $F = F_0, F_1, F_2, \ldots$ such that $F_{i+1}$ differs from $F_i$ by a vertex move, for each $i$, and $F_i \to F'$ pointwise as $i\to \infty$.
\end{df}

\begin{lm}\label{equivalentFlats}
    Suppose that the images $A$, $A'$ of two flats $F$, $F'$ in a quadric complex are at finite Hausdorff distance $h:= d_{\mathrm{Haus}}(A,A')<\infty$. Then $F$ and $F'$ are equivalent. In particular, $d_{\mathrm{Haus}}(A, A')\leq 1$. 
\end{lm}

This allows us to define so called thick flats, which are subcomplexes consisting of all equivalent flats.

\begin{df}\label{thickflat}
    Let $F \colon \mathbb{E}^2_{\square} \to X$ be a flat in a quadric complex. The \emph{thick flat} $\mathrm{Thick}(F)$, is the full subcomplex spanned on all flats at finite Hausdorff distance from $F$.
\end{df}

As a corrolary of their results, Hoda and Munro further obtained the following version of the Flat Torus Theorem applicable to virtually free abelian groups and thick flats.

\begin{thm}\label{virtflat}
    Suppose a virtually $\mathbb{Z}^2$ group $G$ acts properly on a quadric complex $X$.  Then $G$ stabilizes $\mathrm{Thick}(F)$ for a $\mathbb{Z}^2$-invariant flat $F$ in $X$.
\end{thm}

\subsection{Flats in \cftfs complexes}\label{flatcftf}

Analogously to the case of \css complexes we can define flat plane for \cftfs complex.
\begin{df}
A \emph{flat plane} $\mathfrak{F}_{\square}$ is a \cftfs complex homeomorphic to the Euclidean plane such that each closed $2$-cell intersects exactly eight neighboring closed $2$-cells, four by single piece and four by a vertex. Moreover each vertex of valence at least three has valence exactly four. 
For a flat plane $\mathfrak{F}_{\square}$ there is an associated flat plane $\mathfrak{F}'_{\square}$ obtained by ignoring $0$-cells of $\mathfrak{F}_{\square}$ with valence two, so $\mathfrak{F}'_{\square}$ is a regular square tiling of the Euclidean plane. 
\end{df}

Unfortunately, unlike in the case of the nerve of a \css complex, a flat in the quadrization does not necessarily correspond to a flat plane in the original \cftfs complex.
In particular, consider the example of a \cftfs complex and its quadrization shown in Figure \ref{fig:quad}.

\begin{figure}[h]
\begin{center}
\includegraphics[scale=1.5]{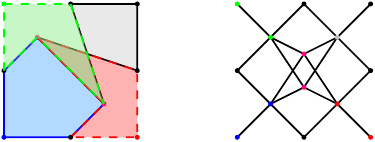}
\end{center}
\caption{}\label{fig:quad}
\end{figure}

It is evident that its quadrization can be derived from a quadrization of four squares by expanding the central vertex into an edge. 

By the action on this complex by free abelian group of rank two we can obtain the following simply connected complex (Figure \ref{fig:plane}).

\begin{figure}[H]
\begin{center}
\includegraphics[scale=1]{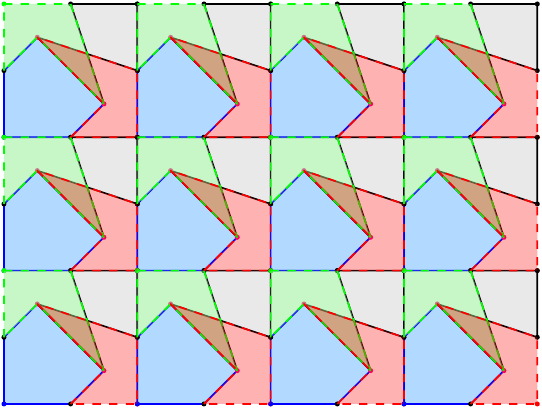}
\end{center}
\caption{}\label{fig:plane}
\end{figure}

We obtain the following proposition:

\begin{pr}
There exists a simply-connected \cftfs small cancellation complex $X$ admitting a proper and cocompact action of a non-cyclic free abelian group $G$ such that $X$ contains no $G$-invariant flat plane.
\end{pr}

Note that as a consequence of CAT(0) Flat Torus Theorem (Theorem \ref{prop:FPT}) we obtain the following corollary:

\begin{cor}
There exists a simply connected \cftfs small cancellation complex $X$ for which no $\mathrm{CAT}(0)$ metric is invariant under the action of $\mathrm{Aut}(X)$. 
\end{cor}

This demonstrates that, unlike the case of \cttss complexes, the question of whether \cftfs groups are CAT(0) cannot be solved by finding a CAT(0) metric on \cftfs complexes.

Nevertheless, the complex $X$ from the example is clearly quasi-isometric to a flat, and, in particular, it can be obtained from a flat by some basic operations.

\begin{df}\label{quasi-flat}
Take set of lines $\mathbb{R}\times\mathbb{Z}\times\{0\}$, $\mathbb{Z}\times\mathbb{R}\times \{1\},  \mathbb{Z}\times \mathbb{Z}\times [0,1]$ as a $1$-skeleton of complex $\mathfrak{X}'$, with vertices at every intersection, see Figure \ref{fig:quasiflat}. 
Assume that every $8$-cycle of the form $(i,j,0),(i+1,j,0),(i+1,j,1),(i+1,j+1,1),(i+1,j+1,0),(i,j+1,0),(i,j+1,1),(i,j,1)$ for $i,j\in \mathbb{Z}$ is filled by a $2$-cell.
Let $\mathfrak{X}$ be obtained from $\mathfrak{X}'$ by collapsing some set of edges of the form $(i,j,0),(i,j,1)$ to vertices.

A thickened-flat plane $\mathfrak{F}_{\octagon}$ is a \cftfs complex homeomorphic to $\mathfrak{X}$.
\end{df}

\begin{figure}[h]
\begin{center}
\includegraphics[scale=1]{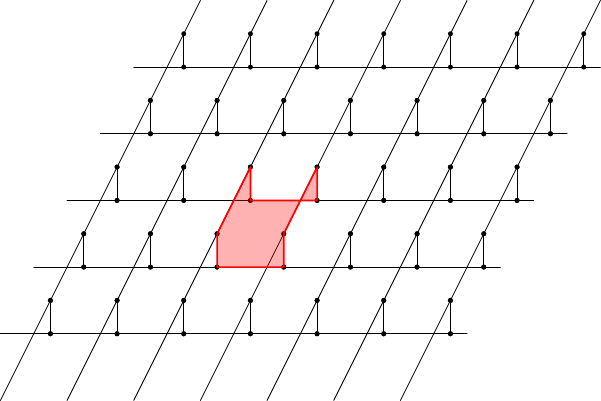}
\end{center}
\caption{Complex $\mathfrak{X}'$ from Definition \ref{quasi-flat} with single $2$-cell marked.}\label{fig:quasiflat}
\end{figure}

Observe that, in particular, the square flat plane is a thickened-flat plane with all vertical edges collapsed. The example of \cftfs complex that is not $\mathrm{CAT}(0)$ given above is also a thickened-flat plane, which can, for instance, be obtained by collapsing all edges except those of the form $(2i,2j,0),(2i,2j,1)$.
In the proof of the following proposition we will again use numberings of the $2$-cells of a complex.

\begin{pr}\label{qf_conditions}
Let $X$ be a \cftfs complex such that each edge is contained in at least two $2$-cells. 
If for each $2$-cell $F\in X$ the following conditions are satisfied: 
\begin{enumerate}
\item $F$ has non-empty intersection with exactly eight other $2$-cells $F_1,\ldots F_8$ (ordered clockwise).
\item $F\cap(F_1\cup F_3\cup F_5\cup F_7)=\partial F$.
\item $F\cap F_2= F_1\cap F_3$, $F\cap F_4= F_3\cap F_5$, $F\cap F_6= F_5\cap F_7$ and $F\cap F_8= F_7\cap F_1$.
\item $F\cap F_i\cap F_j \cap F_k\cap F_l=\emptyset$ for any distinct $i,j,k,l$.
\item For each vertex $v\in \partial F$ of degree $\delta(v)\geq 3$ one of the following holds:
\begin{itemize}
\item $\delta(v)=4$, $\{v\}=F\cap F_i$ for some $i\in \{2,4,6,8\}$ and link of $v$ is an embedded cycle of length $4$ corresponding to cells $F_{i-1},F_{i},F_{i+1}$ and $F$;
\item $\delta(v)=3$, there exists $v'$ such that $\delta(v')=3$, and $(v,v')=F\cap F_i$ for some $i\in \{2,4,6,8\}$. Links of $v$ and $v'$ consists of two embedded cycles of length $2$, cycles in one of them correspond to $F,F_{i-1}$ and $F_{i},F_{i+1}$ and in the other to $F,F_{i+1}$ and $F_{i},F_{i-1}$.
\end{itemize}
\end{enumerate}
then $X$ is a thickened-flat plane.
\end{pr}

\begin{proof}
Consider the complex $\mathfrak{X}'$ from Definition \ref{quasi-flat}.
First, note that $\mathfrak{X}'$ satisfies conditions (1)–(5). Indeed, any $2$-cell $F\in \mathfrak{X}'$ intersects exactly eight other $2$-cells. We enumerate these clockwise, choosing $F_1$ to be a cell whose intersection with $F$ is not merely an edge of the form $(i,j,0),(i,j,1)$. Then, for all even $n$, $F\cap F_n$ is exactly such an edge. Conditions (1)–(4) are then immediate.

Moreover, each vertex $v$ in $\mathfrak{X}'$ has degree $\delta(v)=3$, and its link consists of two embedded cycles, so Condition (5) is also satisfied.

We consider the following numbering of the complex $X$. Let $C_0$ be a $2$-cell of $X$. 
Set $\phi(0)=C_0$, and define 
$$A_0=\{C \in X_2\setminus C_0\ : C\cap C_0 \neq \emptyset\}.$$ 
By Condition (1), $|A_0|=8$.
Next, consider the set 
$$A'_0:=\{C \in A_0\ : \forall C',C''\in (A_0\setminus C),\ C\cap C_0 \not\subset C'\cap C'' \}.$$ 
Note that $|A'_0|= 4$.
Indeed, Condition (3) implies that the intersections of the cells $F_1,F_3,F_5,F_7$ from $A_0$ cover intersections of $F_2,F_4,F_6,F_8$ with $F$ hence $|A'_0|\leq 4$.
Moreover, Condition (3) implies that 
$$F\cap (F_i\cup F_j\cup F_k)\supset F\cap (F_2\cup F_4\cup F_6\cup F_8)$$
for any distinct $i,j,k\in \{1,3,5,7\}$. 
If any of $F_1,F_3,F_5,F_7$ has its intersection with $F$ covered by some other $F_i,F_j$, then $F_i\cup F_j$ is covered by remaining three among $F_1,F_3,F_5,F_7$. 
Consequently, the boundary of $F$ can be covered by three pieces, a contradiction with \cftf.

Choose any $C_1 \in A'_0$, set $\phi(1)=C_1$, and define 
$$A_1=\{C \in X_2\setminus C_1\ : C\cap C_1 \neq \emptyset\}.$$
Since $A_0'$ contains exactly four $2$-cells, and one of them is $C_1$, the intersection $A_0'\setminus A_1$ has at most three $2$-cells. Condition (2) implies that it has at least two cells, and \cftfs implies that it cannot have three; otherwise one of $F_1,F_3,F_5,F_7$ would have intersection with remaining three ones, making it is possible to cover boundary of $C_0$ using only three pieces. Thus $|A_0'\setminus A_1|=2$.
Denote these two cells $C_2,C_3$ and set $\phi(2)=C_2$, $\phi(3)=C_3$.

Inductively, assume that $\phi$ has been defined on a set of cardinality $k+1$, so $\phi(0)=C_0, \ldots, \phi(k)=C_k$.
For each $0\leq i\leq k$ define 
$$A_i=\{C \in X_2\setminus C_i\ : C\cap C_i \neq \emptyset\}$$
and
$$A'_i:=\{C \in A_i\ : \forall C',C''\in (A_i\setminus C),\ C\cap C_i \not\subset C'\cap C'' \}.$$ 
Note that for every $i$ we have $|A_i|=8$ and $|A'_i|=4$.
Let $i$ be the smallest index such that $A'_i\cap \{C_0,\ldots, C_k\}\neq A'_i$.
Then let $j>i$ be the smallest index such that $(A'_i\cap A_{j})\cap \{C_0,\ldots, C_k\}\neq (A'_i\cap A_{j})$. 
Note that $k>1$ implies that $|(A'_i\cap A_{j})\cap \{C_0,\ldots, C_k\}|\geq 1$. Indeed, for each $k>1$ the cell $C_k$ is chosen from $(A'_i\cap A_{j})$, where $i<j\leq k$, thus $C_j\in (A'_i\cap A_{k})$.
We denote by $C_{k+1}$ the remaining cell from $(A'_i\cap A_{j})$ and set $\phi(k+1)=C_{k+1}$. 
Note that this also implies that such numbering is uniquely determined by the choice of the first three cells.

We define numbering $\psi$ of the cells of $\mathfrak{X}'$ in the same way. It is clear that $i,j\in\mathbb{N}$, $\psi(i)\cap\psi(j)\neq\emptyset$ if and only if $\phi(i)\cap\phi(j)\neq\emptyset$ (see Figure \ref{fig:enumeration}).

\begin{figure}[h]
\begin{center}
\includegraphics[scale=0.9]{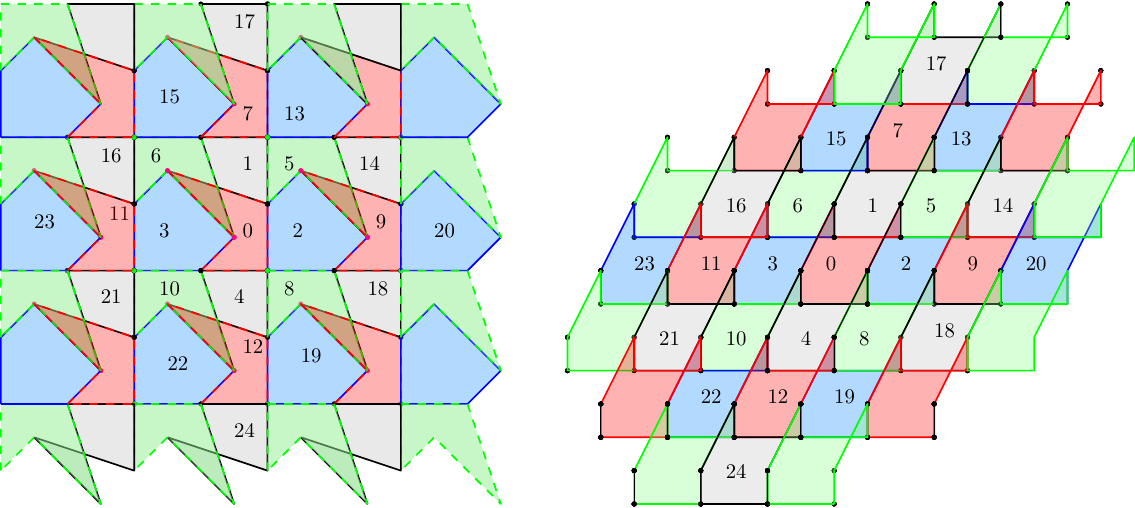}
\end{center}
\caption{Example of numberings $\phi$ and $\psi$ of $X$ and $\mathfrak{X}'$}\label{fig:enumeration}
\end{figure}

Assume that for some $i,j$ the intersection of $\phi(i),\phi(j)$ in $X$ is a single vertex. Take $\phi(i)$ as $F\in X$. Since $X$ satisfies Condition (2), $\phi(j)$ must be one of $F_2,F_4,F_6,F_8$. 
Note that if we take $\psi(i)$ as $\mathfrak{F}\in\mathfrak{X}'$ then $\psi(j)$ is one of corresponding $\mathfrak{F}_2,\mathfrak{F}_4,\mathfrak{F}_6,\mathfrak{F}_8$. Indeed, since $\phi(i)\cap\phi(j)$ is a vertex, Conditions (3) and (4) imply that $\phi(j)=F_n$ has no intersection with $\phi(k):=F_{n\pm 2}$.
But $\mathfrak{F}_n$ for odd $n$ have such intersection, which contradicts the fact that $\psi(i)\cap\psi(j)\neq \emptyset$ iff $\phi(i)\cap\phi(j)\neq \emptyset$.
Thus intersections in $\mathfrak{X}'$, corresponding to any intersection $\phi(i)\cap\phi(j)$ that consists of single vertex, are of the form $(n(i),m(j),0),(n(i),m(j),1)$.

Consequently we can consider complex $\mathfrak{X}$ from Definition \ref{quasi-flat} by collapsing each edge in $\mathfrak{X}'$ for which intersection of $\phi(i),\phi(j)$ in $X$ is a single vertex.
It remains a thickened-flat plane, inherits the numbering $\psi$, and continues to satisfy Conditions (1)–(4). Collapsing edges transforms pairs of degree $3$ vertices into a single degree $4$ vertex, satisfying Condition (5).

All vertices of $X$ are internal.
Indeed, by Condition (5) each vertex of degree greater than $2$ is internal. 
If $v\in \partial F$ has degree $2$, then either it lies in a single $F\cap F_i$, or in two consecutive cells $F_i$ and $F_{i+1}$. Condition (3) ensures it also lies in either $F_{i-1}$ or $F_{i+2}$.
If it is a vertex in the middle of the intersection of the central of those cells with $F$ and is not internal then it has degree greater than $2$, a contradiction.

For any $2$-cell of $X$ there is a natural homeomorphism to the $2$-cell of $\mathfrak{X}$ with the same number.
Conditions (3) and (4) ensure that for any $I\subset \mathbb{N}$ we have $\bigcap\limits_{i\in I} \psi(i)\neq \emptyset$ iff $\bigcap\limits_{i\in I} \phi(i)\neq \emptyset$.
Consequently we can define bijection between $X$ and $\mathfrak{X}$ by taking homeomorphisms between $2$-cells with the same numbers, in a way that agrees on their intersections.
Condition (5) ensure that this map induced a bijection between sets of vertices of $X$ and $\mathfrak{X}$ of degree at least $3$. 
The resulting map is clearly continuous, therefore $X$ is a thickened-flat plane.
\end{proof}

\begin{df}
A thickened-flat plane $\mathfrak{F}_{\octagon}$ is a \emph{thickened-flat} in a \cftfs complex $X$ if it is a subcomplex and, for any pair of $2$-cells  $F_1,F_2\in \mathfrak{F}_{\octagon}$, the gallery distance in $\mathfrak{F}_{\octagon}$ is equal to gallery distance in $X$.
\end{df}

\begin{lm}\label{lm:equivflatdiff}
Let $X$ be a \cftfs complex and $Y$ its quadrization.
Let $F$ be a flat in $Y$ and $F'$ be a flat differing from $F$ by a vertex move (see, Definition \ref{df:equivflats}) for vertices $v,w$ of $Y$. 
Then $v,w$ correspond to $0$-cells of $X$.
\end{lm}

\begin{proof}
It is clear that either both $v$ and $w$ correspond to two $2$-cells or two $0$-cells of $X$.
Consider two $2\times 2$ square grids filling the same $8$-cycle $\{v_1,F_1,v_2,F_2,v_3,F_3,v_4,F_4\}$ in $Y$, with centers $v$ and $w$.
If $v$ and $w$ correspond to $2$-cells $F_v, F_w$, then intersection $F_v\cap F_w\supset \{v_1,v_2,v_3,v_4\}$. By Lemma \ref{lem:intersections} intersection of two $2$-cells is connected, thus intersection of $F_v$ and $F_w$ is a piece. Each $F_i$ contains a piece between $v_i$ and $v_{i+1}$, therefore $F_v\cap (F_1\cup F_2\cup F_3\cup F_4)\subsetneq \partial F_v$.
As a consequence one of $F_i$ contains at least three of $v_i$'s (see Figure \ref{fig:cond2}), and there is an edge in $Y$ between vertices that are at distance $3$ from each other in $F$, contradicting the isometric embedding and the fact that $F$ is a flat.  
\end{proof}

As a consequence we obtain the following.

\begin{cor}\label{cor:thickflat}
Let $X$ be a \cftfs complex and $F$ be a flat in quadrization $Y$ of $X$.
Let $E$ (resp. $\mathrm{Thick(E)}$) be a subcomplex of $X$ consisting of all $2$-cells of $X$ corresponding to vertices of $F$ (resp. $\mathrm{Thick(F)}$).
Then $E=\mathrm{Thick}(E)$.
\end{cor}

\begin{lm}\label{thickthickenedequivalence}
Let $X$ be a \cftfs complex, $Y$ its quadrization and $\mathrm{Thick}(E_Y)$ be a thick flat in $Y$ spanned on flats equivalent to a flat $E_Y$.
Let $\mathrm{Thick}(E)$ be a subcomplex of $X$ consisting of all $2$-cells of $X$ corresponding to vertices of $\mathrm{Thick}(E_Y)$.
Then $\mathrm{Thick}(E)$ is a thickened-flat.
\end{lm}

\begin{proof}
By Corrolary \ref{cor:thickflat}, $\mathrm{Thick}(E)$ is equal to a subcomplex $E$ of $X$ consisting of all $2$-cells of $X$ corresponding to vertices of $E_Y$, so it is enough to show that $E$ is a thickened flat.  

Each gallery in $X$ corresponds to a path of the same length in $Y$. In particular, the gallery distance between two $2$-cells in $X$ is half the distance between the corresponding vertices in $Y$. Therefore, if $E$ is not isometrically embedded, $E_Y$ would not be either, a contradiction.
It remains to show that $E$ satisfies conditions of Proposition \ref{qf_conditions}.

From now on, we denote by $x_i$ the vertices in $Y$ corresponding to the cells $F_i$ in $X$ (indices may be omitted when convenient). We denote by $v_i$ the vertices of $Y$ corresponding to the vertices of $X$, abusing notation by using the same label $v_i$ in both sets.

First, we observe that the intersection of any five $2$-cells of $E$ is empty, which implies Condition (4). Indeed, if a tuple of pairwise intersecting $2$-cells $F_1, F_2, F_3, F_4, F_5$ had a nonempty intersection, then by the Helly property (Lemma \ref{lem:intersections}), their intersection in $E$ would contain at least one vertex. Consequently, there would exist a vertex $v$ in $Y$ adjacent to all of $x_1, \ldots, x_5$.

However, this would imply that the distance between each pair of distinct elements $(x_i, x_j)$ in $Y$ is exactly $2$, which contradicts the possibility of $\mathbb{E}^2_{\square} \hookrightarrow E_Y$ being isometrically embedded in $Y$. To see this, consider balls of radius $2$ around $x_1$ and $x_2$ in the flat $E_Y$, see Figure \ref{fig:quintintersection}. Their intersection contains exactly four vertices from $X_2$, two of which are $x_1$ and $x_2$. Therefore, one of $x_3, x_4, x_5$ would fail to be at distance $2$ from either $x_1$ or $x_2$ in $E_Y$.

\begin{figure}[h]
\begin{center}
\includegraphics[scale=0.7]{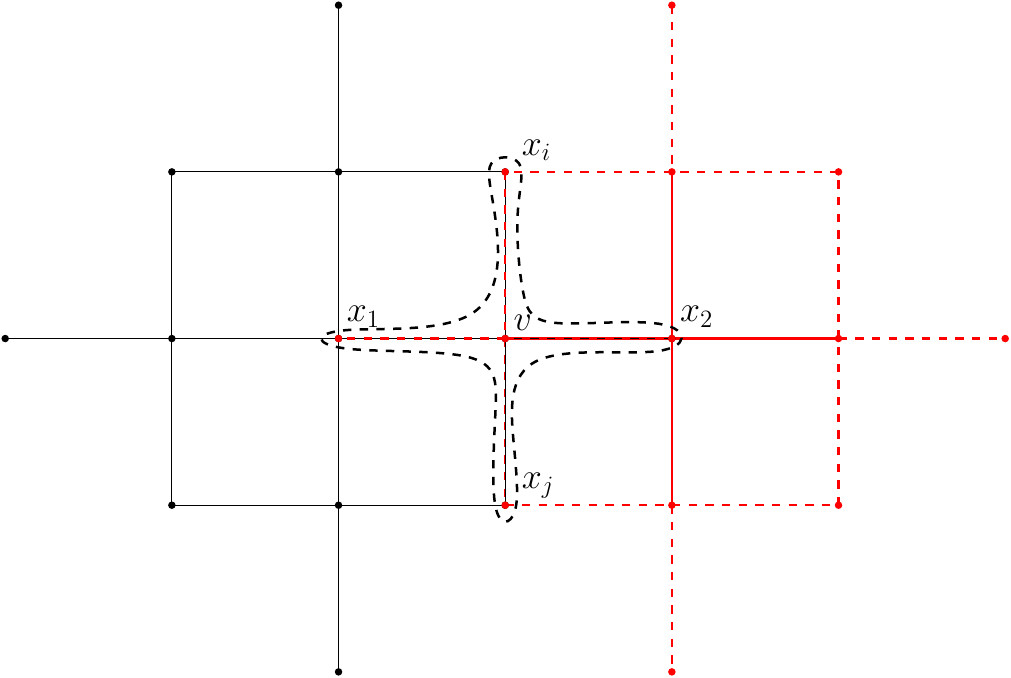}
\end{center}
\caption{Intersection of two balls of radius $2$ around two vertices at distance $2$ from each other in $E_Y$.}\label{fig:quintintersection}
\end{figure}

Now, let $F \in E$ and let $x \in E_Y$ be the corresponding vertex. Observe that exactly four vertices $v_i$ of $E$ belong to $F$. Each of these vertices is at distance $1$ from $x$ in $E$, and since there are only four vertices at distance $1$ from $x$ in the flat $E_Y$, the existence of any additional vertex would again imply that $\mathbb{E}^2_{\square} \hookrightarrow E_Y$ is not isometrically embedded in $Y$. Denote these vertices by $v_1, v_2, v_3, v_4$, ordered clockwise.

Moreover, there are exactly eight cells $F_i$ in $E$ such that $F \cap F_i \neq \emptyset$. Each $2$-cell intersecting $F$ shares at least a vertex in $X$, so $x$ and $x_i$ are at distance $2$ in $Y$. Since there are exactly eight vertices at distance $2$ from $x$ in $E_Y$, any additional cell would violate the isometric embedding of $E_Y$ in $Y$, proving Condition (1).

Denote these eight cells by $F_1, \ldots, F_8$ and their corresponding vertices in $Y$ by $x_1, \ldots, x_8$, ordered clockwise starting from $x_1$, which is adjacent to $v_4$ and $v_1$. Then $x_1, x_3, x_5, x_7$, together with $x$, span squares $[x_1, v_1, x, v_4]$, $[x_3, v_2, x, v_1]$, $[x_5, v_3, x, v_2]$, $[x_7, v_4, x, v_3]$. On the other hand, $x_2, x_4, x_6, x_8$ are each adjacent to $v_{i/2}$, see Figure \ref{fig:ballradius2}.

\begin{figure}[h]
\begin{center}
\includegraphics[scale=0.7]{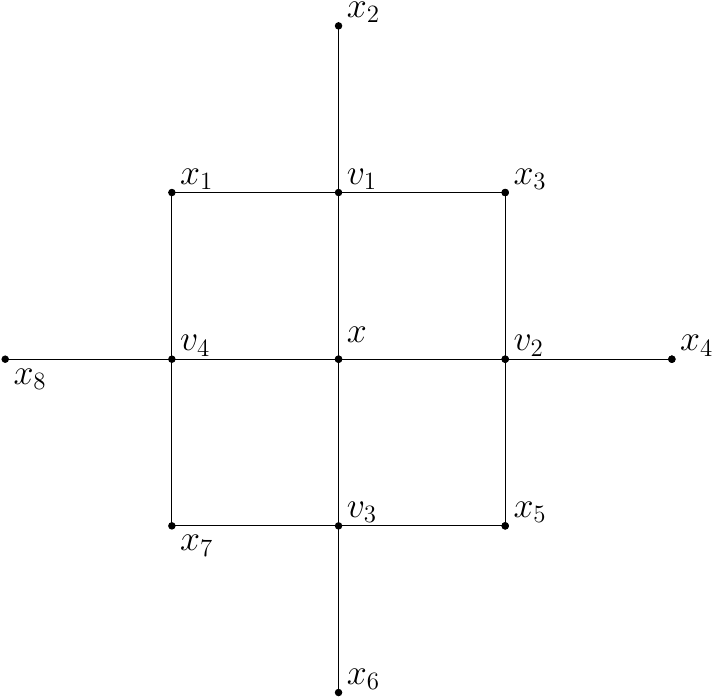}
\end{center}
\caption{Vertices $x_1\ldots x_8$ and $v_1,\ldots v_4$ around vertex $x$ in $E_Y$}\label{fig:ballradius2}
\end{figure}

Observe that $(F_1 \cup F_3 \cup F_5 \cup F_7) \cap F = \partial F$. Indeed, if the pieces $(v_1, v_2), (v_2, v_3), (v_3, v_4), (v_4, v_1)$ do not cover $\partial F$, then these pieces form a path in $F$, see Figure \ref{fig:cond2}. In this case, there exists $v_j \in (v_i, v_{i+1})$ for some $i$, implying $v_j$ is adjacent in $Y$ to the same $x_s$ as $(v_i, v_{i+1})$. But then $x_s$ and $v_j$ are at distance $3$ in $E_Y$, contradicting the isometric embedding. Thus, Condition (2) holds. Moreover it is clear that every edge in $E$ belongs to at least two $2$-cells.

\begin{figure}[h]
\begin{center}
\includegraphics[scale=0.7]{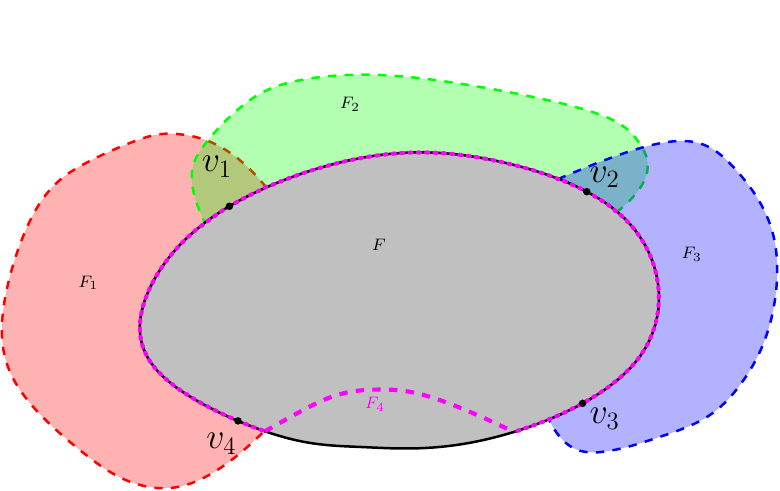}
\end{center}
\caption{ If pieces $(v_1,v_2), (v_2,v_3), (v_3,v_4),(v_4,v_1)$ do not cover boundary of $F$ then one of $F_i$ has to contain at least three of $v_i$'s.}\label{fig:cond2}
\end{figure}

Observe that $F\cap F_i\subset F_{i-1}\cap F_{i+1}$ for $i=2,4,6,8$.
This condition is obvious if $F\cap F_i=\{v_{\frac{i}{2}}\}$. If this intersection is a piece, then there exists additional vertex $v'\in Y$ spanning a square with $x$, $v_{\frac{i}{2}}$ and $x_i$. Thus there exist two six cycles spanned by three squares in $Y$, see for Figure \ref{fig:cond3}. As a consequence, there is an edge either between $v'$ and $x_{i\pm 1}$, $x_i$ and $v_{\frac{i}{2}\pm 1}$ or between $x$ and a vertex in $E_Y$ that is at distance three from it. Since $E_Y$ is isometrically embedded, only first configuration is possible, confirming Condition (3).

\begin{figure}[h]
\begin{center}
\includegraphics[scale=0.7]{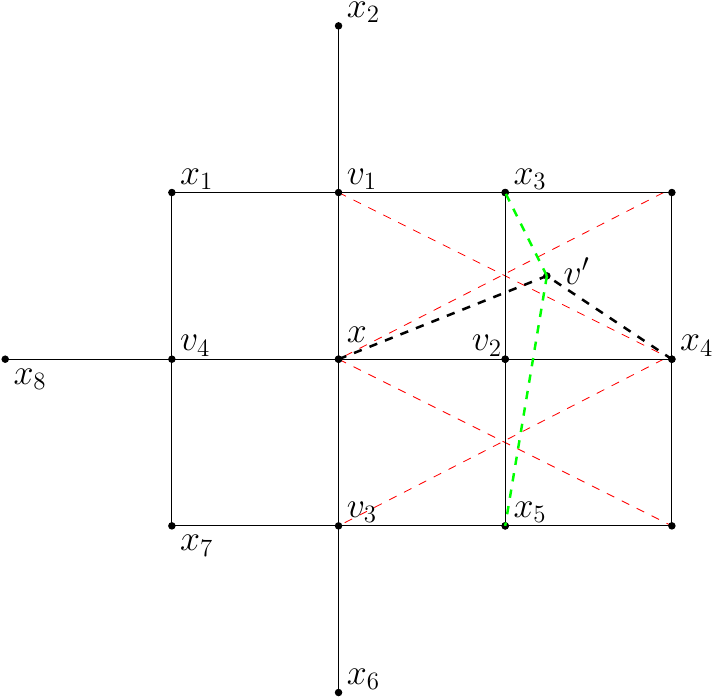}
\end{center}
\caption{Existence of square $[x, v_{2}, x_4,v']$ implies existence of two 6-cycles in $Y$. Since either of red edges contradicts the fact that $E_Y$ is isometrically embeded, both green edges belong to $Y$.}\label{fig:cond3}
\end{figure}

Finally, we verify that all vertices of $E$ satisfy Condition (5). Let $v$ be a vertex of $E$. There exists $F'_1$ such that $v \in \partial F'_1$, and since each cell is fully covered by pieces, another cell $F'_2$ such that $v \in F'_1 \cap F'_2$. If there is no $F'_3$ with $v \in F'_1 \cap F'_2 \cap F'_3$, then $v$ has degree $2$, since each edge adjacent to $v$ lies in $F'_1 \cap F'_2$, and $2$-cells of simply connected \cftfs complexes have no self-intersections.

Thus there is $F'_3\in E$ such that $v\in F'_1\cap F'_2\cap F'_3$. Observe that it implies that there is also $F'_4$ such that $v\in F'_1\cap F'_2\cap F'_3\cap F'_4$. 
Indeed, if there is no such cell, then since boundary of each cell is fully covered by pieces, each edge adjacent to $v$ belongs to one of $F'_1\cap F'_2$, $F'_1\cap F'_3$, $F'_2\cap F'_3$. Since each cell can contain at most two edges adjacent to $v$, the degree of $v$ cannot exceed $3$. If it is equal to $3$ then there is a reduced disc diagram in $X$ containing internal vertex of degree $3$, a contradiction to T($4$) condition.
Thus degree of $x$ is exactly $2$. In $E$ there are eight cells $F_1,\ldots F_8$ with non-empty intersection with $F'_1$. Among these, there is $F_i=F'_2$ and $F_j=F'_3$. It is clear that difference between $i$ and $j$ is at most $2$ (modulo $8$) if both indices are odd, or at most $1$ if either index is even, as otherwise their intersection would be empty. 
For any even $i$ we have $F'_1\cap F_i\subset F_{i-1}\cap F_{i+1}$, which ensures existence of $F'_4$.

By already proven Condition (4) the intersection of each tuple of five distinct cells is empty. Therefore, each vertex belong either to exactly two cells or exactly four cells.
It follows that $v$ cannot have degree higher than $4$. We already noted that if $v$ belongs only to two $2$-cells, then it has degree $2$.
If $v$ belongs to four $2$-cells then by the fact that each edge belongs to intersection of two $2$-cells and the fact that every cell can contain at most two edges adjacent to $v$, the degree cannot exceed $4$.

Now consider the link $X_v$ of any vertex $v\in E$. There is always an  embbeded cycle in $X_v$. Indeed, link $X_v$ consist either of four edges and at most four vertices or of two edges and two vertices. Since none of these edges is a loop, in either case there exists an embedded cycle of length at least $2$.
Condition T($4$) forbids embedded cycle of length $3$, so it has length either $2$ or $4$.

If $v$ has degree three, the embedded cycle clearly has length $2$, formed by some pair from $F'_1,F'_2,F'_3,F'_4$. The remaining pair also forms a cycle of length $2$. Indeed, the first pair contains two of three edges adjacent to $v$, and the remaining cells must contain the third one. Since the boundaries of each cell are fully covered, the two remaining cells also contain the remaining edges. They must share the same edge, as otherwise one obtains a reduced disc diagram with internal vertex of degree three, contradicting \cftf. Thus they also form an embedded cycle of length $2$ in $X_v$.

If $v$ has degree $4$ and an embedded cycle in $X_v$ is not of length $4$, there must be two embedded cycles of length $2$, as each vertex of $X_v$ belongs to two edges.
Without loss of generality assume that one cycle corresponds to $F'_1,F'_2$ and the other to $F'_3, F'_4$. In this case $F'_1$ shares at least two edges with $F'_2$ and only the vertex $v$ with both $F'_3$ and $F'_4$.
Since $F'_2,F'_3,F'_4$ are three consecutive (not necsessarily ordered) cells intersecting $F'_1$, either $F'_1\cap F'_2\subset F'_3\cap F'_4$, or one of $F'_3,F'_4$ shares at least an edge with $F'_1$. In either case a contradiction.

It follows from Proposition \ref{qf_conditions} that $E$ is a thickened-flat, so $\mathrm{Thick}(E)=E$ is a thickened flat.

\end{proof}

This allows us to prove Theorem \ref{flattorus44}.

\begin{proof}[Proof of Theorem \ref{flattorus44}]
Let $G$ be a virtually non-cyclic free abelian group and $H$ be its non-cyclic free abelian subgroup of finite index. Since $G$ acts metrically properly (with respect to gallery metric) on $X$, it acts metrically properly (with respect to standard graph metric on $1$-skeleton) on its quadrization $Y$.
By Quadric Flat Torus Theorem of Hoda and Munro (Theorem \ref{quadflat}), if a free-abelian group $H$ of rank $n\geq 2$ acts metrically properly on a quadric complex then $H\cong \mathbb{Z}^2$ and there is a $H$-invariant flat $E_Y$ in $Y$.
By Lemma \ref{thickthickenedequivalence} an $H$-invariant subcomplex $E$ corresponding to $E_Y$ in $X$ is a thickened-flat.

Moreover, by Theorem \ref{virtflat} $\mathrm{Thick}(E_Y)$ is $G$-invariant, and as a consequence of Corollary \ref{cor:thickflat} this holds for $E$ as well.

\end{proof}

\bibliographystyle{alpha}
\bibliography{references}{}
\end{document}